\documentclass[12pt]{article}
\usepackage{amsmath}
\usepackage{amsfonts}
\usepackage{amssymb}
\usepackage{amsthm}
\usepackage[left=2cm,right=2cm]{geometry}
\usepackage{mathtools}
\usepackage{enumitem}
\usepackage{tikz}
\usepackage{tikz-cd}

\allowdisplaybreaks

\newtheorem{theorem}{Theorem}[section]
\newtheorem{lemma}[theorem]{Lemma}

\newtheorem{corollary}[theorem]{Corollary}

\theoremstyle{definition}

\newtheorem{definition}[theorem]{Definition}

\newtheorem{remark}[theorem]{Remark}

\newcommand{\Z}{\mathbb{Z}}
\newcommand{\N}{\mathbb{N}}
\newcommand{\R}{\mathbb{R}}

\newcommand{\T}{\mathbb{T}}

\renewcommand{\tilde}{\widetilde}
\renewcommand{\hat}{\widehat}
\newcommand{\Ind}{\textup{Ind}}

\renewcommand{\tilde}{\widetilde}

\def\ip<#1,#2>{\left\langle #1,#2 \right\rangle}

\newcommand{\vertiii}[1]{{\left\vert\kern-0.25ex\left\vert\kern-0.25ex\left\vert #1 
		\right\vert\kern-0.25ex\right\vert\kern-0.25ex\right\vert}}

\title{Constructing Well-bounded Operators not of type (B) on a Class of Inductive Limits}

\author{Alan Stoneham}

\date{}

\begin{document}

	\maketitle
	
	\begin{abstract}
		Well-bounded operators are linear operators on a Banach space $X$ that have an $AC[a,b]$ functional calculus for some interval $[a,b]$. A well-bounded operator is of type (B) if it can be written as an integral against a spectral family of projections, and this is always the case when $X$ is reflexive. There are many examples of well-bounded operators on non-reflexive spaces that are not of type (B), and it is open whether there is a non-reflexive Banach space upon which every well-bounded operator is of type (B). It was suggested in \cite{comp-wb} that the spaces constructed in \cite{Pisier} could provide an example of such a space. In this paper, it will be shown that on a class of Banach spaces containing the spaces from \cite{Pisier}, there is always a well-bounded operator not of type (B).
	\end{abstract}

\numberwithin{equation}{section}
	
	\section{Introduction}
	
	The class of well-bounded operators was introduced by Smart to generalise the spectral theory for self-adjoint operators to a theory which could accommodate operators on Banach spaces whose spectral expansions converged conditionally. Well-bounded operators are bounded linear operators that possess a functional calculus for the absolutely continuous functions on some compact interval $[a,b]$. It was shown by Smart and Ringrose \cite{Smart,Ringrose1} that a well-bounded operator on a reflexive Banach space $X$ could always be written as an integral with respect to a unique `spectral family of projections' on $X$. However, on non-reflexive Banach spaces, there are simple examples of well-bounded operators which could not be written as such an integral, such as $T$ acting on $C[0,1]$ defined by $Tx(t)=tx(t)$. Ringrose \cite{Ringrose2} later extended this theory to the non-reflexive setting by considering projections on the dual space $X^*$ rather than $X$, and here the family of projections may no longer be unique. 
	
	Berkson and Dowson \cite{BD} considered classes of well-bounded operators whose projections possessed certain properties. They said that a well-bounded operator is of type (B) if it possessed a spectral family of projections, and they with Spain \cite{Spain-wb} classified the type (B) well-bounded operators as those whose functional calculus is weakly compact. It is still an open problem whether there is a non-reflexive Banach space where all well-bounded operators are of type (B). Current results, which will be discussed in Section~\ref{sec-background}, show that such a Banach space would have to be rather exotic.
	
	Much like the spectral theorem for compact self-adjoint operators, it is known from \cite{CDnr1} that a compact well-bounded operator $T$ can always be expressed as
	\begin{align}\label{eqn-comp}
		T=\sum_{j=1}^\infty \lambda_j P_j
	\end{align}
	for a sequence of real numbers $\{\lambda_j\}_{j=1}^\infty$ converging monotonely in absolute value to 0 and $\{P_j\}_{j=1}^\infty$ a sequence of uniformly bounded, mutually disjoint projections of finite rank. This identification provides a simple way to construct well-bounded operators with various properties. An in-depth discussion of compact well-bounded operators was undertaken by Cheng and Doust \cite{comp-wb} where compact well-bounded operators of type (B) were classified. Moreover, they showed that from a compact well-bounded operator of infinite rank, one can find an increasing sequence of uniformly bounded finite rank projections. That is, a sequence of finite rank projections $\{Q_n\}_{n=1}^\infty$ such that $\sup\limits_n\|Q_n\|<\infty$, $Q_m\neq Q_n$ and $Q_mQ_n=Q_nQ_m = Q_{\min(m,n)}$ for all $m,n\in\Z^+$. 
	
	In \cite{Pisier}, Pisier showed that every Banach space $E$ of cotype 2 is isometric to a subspace of a Banach space $X_E$ also of cotype 2 satisfying, among other things, the following property:
	\begin{center}
		there is a constant $C_E>0$ such that if $P$ is a finite rank projection on $X_E$, then
	\begin{align}\label{pisier-property}
		\|P\| \geq C_E\sqrt{\text{rank}(P)}.
	\end{align}
	\end{center}
	It is well-known (for example, \cite{Pisier-book} Theorem 1.14) that given an $n$-dimensional subspace $Y$ of a Banach space $X$, there is a finite rank projection $P:X\to Y$ with $\|P\|\leq \sqrt{n}$, and so Pisier's spaces exhibit extreme behaviour concerning finite rank projections.
	As a consequence of \eqref{pisier-property}, there can be no increasing sequences of uniformly bounded finite rank projections. In particular, the only compact well-bounded operators are the ones of finite rank, which are all of type (B). Hence it was suggested in \cite{comp-wb} that it may be possible for all well-bounded operators on these spaces to be of type (B).
	
	The space $X_E$ is constructed by taking a Banach space $E=E_0$ and forming a particular sequence of Banach spaces $E_0,E_1,E_2,\dots$ such that each space is isometric to a subspace of the next, the construction of which is detailed in Section~\ref{pisier-construct}. The inductive limit of this sequence of Banach spaces is $X_E$, upon which \eqref{pisier-property} is obtained when $E$ is of cotype 2. In Section~\ref{sec-results}, it will be shown that for every Banach space $E$, there is a well-bounded operator on $X_E$ not of type (B).


	\section{Background}\label{sec-background}
	
	In this section, the necessary definitions and results for this article regarding well-bounded operators will be presented, followed by a detailing of the construction of the space $X_E$ from a Banach space $E$. Throughout the following, $X$ and $E$ will denote Banach spaces and $B(X,E)$ will denote the set of bounded linear operators from $X$ to $E$ with $B(X)=B(X,X)$. The continuous dual of $X$ will be denoted by $X^*$ and the adjoint of $T:X\to E$ denoted by $T^*$. A `projection' $P$ will refer to a bounded linear operator acting on a Banach space such that $P^2=P$. Lastly, SOT will refer to the strong operator topology.
	
	\subsection{Well-bounded Operators}
	
	An operator $T\in B(X)$ is \textbf{well-bounded} if there exists an interval $[a,b]$ and a $K>0$ such that
		\begin{align*}
			\|p(T)\| \leq K\left( |p(a)| + \int_a^b|p'(t)|\,dt \right)
		\end{align*}
		for all polynomials $p$. As the polynomials are dense in $AC[a,b]$, the expression $f(T)$ can be made sense of for $f\in AC[a,b]$, and the mapping $f\mapsto f(T)$ defines a continuous Banach algebra homomorphism. That is, the operator $T$ has a functional calculus for $AC[a,b]$. A well-bounded operator is \textbf{of type (B)} if it can be written as an integral against a spectral family of projections. 
	
	\begin{definition}\label{spectral family}
		A \textbf{spectral family of projections} is a projection-valued function $E:\R\to B(X)$ satisfying the following five conditions:

		\begin{enumerate}[label=(\roman*)]
			\item there exist $a,b\in\R$ such that $E(\lambda)=0$ for all $\lambda<a$ and $E(\lambda)=I$ for all $\lambda\geq b$;
			\item $\sup\limits_\lambda \|E(\lambda)\|<\infty$;
			\item $E(a)=0$ and $E(b)=I$;
			\item $E(\lambda)E(\mu) = E(\min(\lambda,\mu))$ for all $\lambda,\mu\in\R$;
			\item $E$ is right-continuous and has left limits in the SOT. That is, for all $\mu\in\R$ and $x\in X$, $\lim\limits_{\lambda\to\mu^+}E(\lambda)x=E(\mu)x$ and $\lim\limits_{\lambda\to\mu^-}E(\lambda)x$ exists.
		\end{enumerate}
	\end{definition}
	
	It was shown in \cite{BD} that a well-bounded operator $T$ acting on $X$ being of type (B) is equivalent to $f\mapsto f(T)x$ being a weakly compact mapping of $AC[a,b]$ into $X$ for each $x\in X$. Hence on reflexive spaces, all well-bounded operators are of type (B).

	One can integrate a function of bounded variation $g\in BV[a,b]$ against a spectral family $E$ via a Riemann-Stieltjes integral. For a partition $\mathcal{P} = \{t_j\}_{j=0}^n$ of $[a,b]$, partial sums of the form
	\begin{align*}
		g(a)E(a) + \sum_{j=1}^n g(t_j)\big( E(t_{j})-E(t_{j-1}) \big)
	\end{align*} 
	converge in the SOT with respect to refinement of $\mathcal{P}$, and the resulting operator is usually denoted in literature by
	\begin{align*}
		\int_{[a,b]}^\oplus g(\lambda)\, dE(\lambda).
	\end{align*}
	 The spectral theorem for well-bounded operators of type (B) (see Ch.~17 of \cite{Dowson}) states that there is a bijective correspondence between well-bounded operators $T$ of type (B)  and spectral families $E$ given by
	\begin{align*}
		T=\int_{[a,b]}^\oplus \lambda\, dE(\lambda).
	\end{align*}

One of the simplest ways to construct well-bounded operators is via increasing sequences of uniformly bounded projections.	 Moreover, it is simple to identify when these well-bounded operators are of type (B).
	
	\begin{definition}
		A sequence of projections $\{P_k\}_{k=1}^\infty$ on a Banach space is said to be \textbf{increasing} if $P_k P_\ell=P_\ell P_k=P_{\min(k,\ell)}$ and $P_k\neq P_\ell$ for all $k,\ell\in\Z^+$. 
	\end{definition}

	\begin{theorem}\textup{(\cite{CDnr1}, Theorem 3.1 and Proposition 4.1.)}\label{thm-proj-B}
		Suppose that $\{\lambda_n\}_{n=1}^\infty$ is a convergent sequence of increasing real numbers with limit $L$ and $\{P_n\}_{n=1}^\infty$ is an increasing sequence of uniformly bounded projections on a Banach space $X$. 
		 Then $T= I - \sum\limits_{k=1}^\infty (\lambda_{k+1}-\lambda_k)P_k$ defines a well-bounded operator on $X$. Moreover $T$ is of type (B) if and only if $\lim\limits_{n\to\infty} P_n$ exists in the SOT. 
	\end{theorem}
	
	It is currently an open problem whether or not there is a non-reflexive Banach space upon which every well-bounded operators is of type (B). Such a space must be rather unusual as most classical spaces are known to have a well-bounded operator not of type (B) due to the following results.

	\begin{theorem}\label{c0 theorem}\textup{(\cite{DdL}, Theorem 4.4)}
		Suppose that a Banach space $X$ contains a subspace isomorphic to $c_0$. Then there is a well-bounded operator on $X$ not of type (B).
	\end{theorem}

	\begin{theorem}\label{non-refl-theorem}
		\textup{(\cite{CDnr1}, Theorem 4.2)} If a Banach space $X$ contains a complemented non-reflexive subspace with a Schauder basis, then there is a well-bounded operator on $X$ not of type (B).
	\end{theorem}

	Due to the property stated in \eqref{pisier-property}, the spaces constructed by Pisier cannot have a Schauder basis and hence the second of the above theorems does not apply directly. Moreover, in the case of $E$ being of cotype 2, the corresponding space $X_E$ is also of cotype 2 and hence cannot possibly contain a subspace isomorphic to $c_0$. However, it will be shown that $X_E$ always contains an increasing sequence of uniformly bounded projections of infinite rank that do not converge in the SOT, and hence there is a well-bounded operator on $X_E$ not of type (B) by Theorem~\ref{thm-proj-B}. To construct these projections, it will be necessary to understand how $X_E$ is constructed.

	\subsection{The Construction of $X_E$}\label{pisier-construct}

	The first ingredient for the construction of $X_E$ is inductive limits of Banach spaces. Given a sequence of Banach spaces $E_0,E_1,E_2,\dots$ and linear isometries $j_n:E_n\to E_{n+1}$ for each $n\in\N$, one may construct the (Banach) \textbf{inductive limit} $\Ind(E_n,j_n)$. Consider the subspace of $\prod_{n=0}^{\infty} E_n$ consisting of sequences $(x_n)_{n=0}^\infty$ such that $j_nx_n = x_{n+1}$ for all $n$ large enough, which will be denoted by $\mathcal{X}$. (The subspace $\mathcal{X}$ may be thought of as the set of sequences that are eventually `constant'.) A semi-norm $\rho$ may be defined on $\mathcal{X}$ by $\rho(x) = \lim\limits_{n\to\infty} \|x_n\|$ and the inductive limit is then the completion of $\mathcal{X}/\ker\rho$. Note that one may think of $\mathcal{X}/\ker\rho$ as the collection of cosets of sequences which are eventually the same.
	
	For each $n\in\N$, there is a linear isometry $i_n$ from $E_n$ to the inductive limit given by mapping $x\in E_n$ to
	\begin{align}\label{ind-lim-i_n}
		(0,\dots,0,x,j_nx,j_{n+1}j_nx,\dots)+\ker\rho,
	\end{align}
	where there are $n$ zeros before $x$ in the sequence. As these maps satisfy $i_{n+1}j_n=i_n$ for each $n$, $E_n$ is isometric to the subspace $i_n(E_n)$ of the inductive limit. Hence $\mathcal{X}/\ker\rho$ may be identified with $\bigcup_{n=0}^\infty i_n(E_n)$. (For this reason, it is often easiest to think of the inductive limit as being  $\overline{\bigcup_{n=0}^\infty E_n}$.)
	

	The next ingredient is to construct the particular sequence of Banach spaces from some base space $E=E_0$. The construction involves iterating the following idea which Pisier credits to Kisljakov \cite{Kisljakov}. Let $E$ and $B$ be Banach spaces, $S$ a closed subspace of $B$, and $u:S\to E$ a bounded linear operator with $\|u\|\leq 1$. Then there exists a Banach space $\tilde{E}$, a linear isometry $j:E\to\tilde{E}$ and a bounded linear operator $\tilde{u}:B\to \tilde{E}$ with $\|\tilde{u}\| \leq 1$ and $\tilde{u}|_{S} = ju$.
	
	\begin{figure}[h]
		\centering
		
		$	\begin{tikzcd}[row sep=huge, column sep=huge,every label/.append style = {font=\small}]
			B \arrow[r,"\tilde{u}"] & \tilde{E} &  \arrow[l, "\pi"'] B\oplus E\\
			S \arrow[ u, hook] \arrow[r, "u"] & E \arrow[u,hook, "j"] \arrow[ur,hook]
		\end{tikzcd}$
		\caption{A commutative diagram illustrating the construction of $\tilde{E}$ from $E$.}\label{diagram}
	\end{figure}
	
	The space $\tilde{E}$ is constructed by taking the quotient of $B\oplus E$, equipped with the norm $\|(x,e)\|=\|x\|+\|e\|$, by the subspace $\{(s,-us)\}_{s\in S}$. If $\pi$ denotes the quotient map from $B\oplus E$ to $\tilde{E}$, then $\tilde{u}$ is defined by $\tilde{u}x = \pi( x,0)$ and the linear isometry $j$ of $E$ into $\tilde{E}$ is given by $j e=\pi(0,e)$. Indeed, since $\|us\|\leq \|s\|$ for all $s\in S$,
	\begin{align*}
	   \|je\| \leq \|e\| \leq \inf_{s\in S}( \|e\| + \|s\|-\|us\| ) \leq \inf_{s\in S} (\|s\|+\|e-us\|) = \|je\|,
	\end{align*} 
	and so $\|je\| = \|e\|$.
	
	\begin{remark}\label{rem-quotient isometry}
		A property of the construction is that $B/S$ is isometric to $\tilde{E}/jE$ under the mapping $x+S\mapsto \pi(x,0)+jE$. Firstly, this map is well-defined since for all $s\in S$, $\pi(s,0) = \pi(0,us)$ and
		 \begin{align*}
			\pi(x+s,0) + jE = \pi(x,0)+jus +jE = \pi(x,0) + jE.
		\end{align*}
		Secondly, it is surjective since $ \pi(x,e)+jE= \pi(x,0)+jE$ for all $e\in E$. Lastly, it is isometric since
		\begin{align*}
			\|\pi(x,0)+jE\| =  \inf_{e\in E} \|\pi(x,e)\| = \inf_{e\in E} \inf_{s\in S} \| (x+s,e-us)\| = \inf_{s\in S} \|x+s\| = \|x+S\|.
		\end{align*}
	\end{remark}
	
%

	The final ingredients for Pisier's construction involve specific choices of $B$, $S$ and $u$. These will be built from the following three Banach spaces and choices of closed subspaces. Let $\T$ denote the complex unit circle, $D=\{-1,1\}^\N$, and $\mu$ be the normalised Haar measure on $D$ induced by the group structure given by entrywise multiplication. Set
	\begin{itemize}
		\item $B_1=L^1(D,\mu)$ and $S_1=\overline{\text{span}}\{\varepsilon_k\}_{k=0}^\infty$ where $\varepsilon_k$ is the $k$th coordinate functional on $D$ defined by $\varepsilon_k((d_i)_{i=0}^\infty) = d_k$;
		\item $B_2 = L^1(\T)$ equipped with the normalised Lebesgue measure and $S_2=H^1(\T)$, where $H^1(\T)$ is the subspace of $L^1(\T)$ consisting of functions whose negative Fourier coefficients are zero;
		\item $B_3 = L^1(\T)/H^1(\T)$ and $S_3 = \overline{\text{span}}\{\xi_k\}_{k=1}^\infty$, where $\xi_k$ is the coset of $z\mapsto z^{-3^k}$.
	\end{itemize}
	Now, for a collection of Banach spaces $(X_i)_{i\in I}$, define the notation
\begin{align*}
	\ell^1(X_i)_{i\in I} = \left\{ (x_i)_{i\in I} \in \prod_{i\in I} X_i \,:\, \| (x_i)_{i\in I} \| = \sum_{i\in I} \|x_i\| < \infty \right\}.
\end{align*}
 Let $F(S_i,E)$ denote the set of non-zero finite rank operators from $S_i$ to $E$ and let $A=\bigcup_{i=1}^3 F(S_i,E)$. For each $\alpha\in A$, set $B_\alpha=B_i$ and $S_\alpha=S_i$ whenever $\alpha\in F(S_i,E)$. The spaces $B$ and $S$ are defined as $B=\ell^1 (B_\alpha)_{\alpha\in A}$, $S=\ell^1 (S_\alpha)_{\alpha\in A}$, and the operator $u:S\to E$ defined by
	\begin{align*}
		u((s_\alpha)_{\alpha\in A} ) = \sum_{\alpha\in A} \frac{\alpha s_\alpha}{ \|\alpha\|},
	\end{align*}
	which satisfies $\|u\|\leq 1$.
	
	By iterating Kisljakov's construction from Figure~\ref{diagram}, one obtains a sequence of Banach spaces $E_0,E_1,E_2,\dots$ with $E_0=E$, $E_{n+1}=\tilde{E_n}$ and linear isometries $j_n:E_n\to E_{n+1}$ for all $n\in\N$. The space $X_E$ is the inductive limit $\Ind(E_n,j_n)$.
	
	Obviously $X_E$ is non-reflexive if $E$ is non-reflexive, and comments made by Pisier in chapter 10e of his monograph \cite{Pisier-book} imply that whenever $X_E$ contains a subspace $M$ that is isomorphic to $\ell^2$ (so that $X_E^*$ admits $\ell^2$ as a quotient via $X_E^*/M^\perp \cong M^*\cong\ell^2$), the properties of $X_E$ and $X_E^*$ ensure that $X_E^*$ contains a subspace isomorphic to $\ell^1$, and hence cannot be reflexive. 
	To conclude this section, it will be demonstrated that $X_E$ is always non-reflexive.
	
	\begin{lemma}\label{lem-quotient isom}
		Let $(B_i)_{i\in I}$ be a collection of Banach spaces and, for each $i\in I$, let $S_i$ be a closed subspace of $B_i$. Then $\ell^1(B_i)_{i\in I}/\ell^1(S_i)_{i\in I}$ is isometrically isomorphic to $\ell^1(B_i/S_i)_{i\in I}$ under the mapping
		\begin{align*}
			(x_i)_{i\in I} + \ell^1(S_i)_{i\in I} \mapsto ( x_i + S_i)_{i\in I}.
		\end{align*}
	\end{lemma}
	The first isomorphism theorem yields that the above mapping is a linear isomorphism, after which it is not difficult to show that it is also an isometry.
	
	Recall that reflexivity is a three-space property (TSP). That is, for a closed subspace $Y$ of a Banach space $X$, one has that $X$ is reflexive if and only if both $Y$ and $X/Y$ are reflexive.
	
	\begin{theorem}
		For all Banach spaces $E$, the space $X_E$ is non-reflexive.
	\end{theorem}

	\begin{proof}
		Without loss, suppose that $E$ is a reflexive Banach space. By Remark~\ref{rem-quotient isometry}, $\tilde{E}/jE$ is isometric to $B/S$, which is isometric to $\ell^1(B_\alpha/S_\alpha)_{\alpha \in A}$ by Lemma~\ref{lem-quotient isom}, which clearly contains $B_1/S_1$ as a closed subspace. It is well-known that the Khintchine inequalities imply that $S_1$ is isomorphic to $\ell^2$. As reflexivity is a TSP, it must be that $B_1/S_1$, and hence $\tilde{E}/jE$, are non-reflexive. Again appealing to reflexivity being a TSP yields that $\tilde{E}$, and hence $X_E$, are non-reflexive.
	\end{proof}
	

	\section{Main Results}\label{sec-results}
	
	\subsection{Extending Projections from $E$ to $\tilde{E}$}
	
	The core idea behind constructing projections on $X_E$ is being able to lift a projection $P$ defined on $E$ to a projection $\tilde{P}$ defined on $\tilde{E}$.
	
	\begin{lemma}\label{lem-proj-ext}
		Let $E$ be a Banach space and $P$ a projection on $E$. Then there is a projection $\tilde{P}$ on $\tilde{E}$ such that $\tilde{P}|_{jE} = jP$ and $\|\tilde{P}\|=\|P\|$.
	\end{lemma}
	
	\begin{proof}
		Recalling that $B=\ell^1 (B_\alpha)_{\alpha\in A}$, define $Q:B\to B$ by 
		\begin{align*}
			(Qx)_{\alpha} = \sum_{\beta\,:\, \alpha=P\beta} x_\beta \frac{\|P\beta\|}{\|\beta\|} \qquad \forall \alpha\in A,
		\end{align*}  
		where an empty sum is interpreted to be zero. The definition of $Q$ is worth some elaboration. To obtain the $\alpha$th coordinate of $Qx$, one looks at the indices/operators $\beta\in F(S_\alpha,E)$ such that $\alpha=P\beta$ and adds the corresponding coordinates $x_\beta$ weighted by $\|P\beta\| / \|\beta\|$. Consequently, $(Qx)_\alpha = 0$ unless $\alpha\in PA = \{ P\beta \,:\, \beta\in A \}$. It is now simple to check that $Q$ is linear and $\|Q\|\leq \|P\|$. To show that $Q$ is a projection, first identify $B^*$ with $\ell^\infty(B_\alpha^*)_{\alpha\in A}$, where
		\begin{align*}
			\ell^\infty(B_\alpha^*)_{\alpha\in A} = \left\{ (\psi_\alpha)_{\alpha\in A}\in \prod_{\alpha\in A} B_\alpha^* \,:\, \sup_{\alpha\in A}\|\psi_\alpha\|<\infty \right\}.
		\end{align*}
		Then one sees that $Q^*$, and hence $Q$, are projections since for all $\psi \in \ell^\infty(B_\alpha^*)_{\alpha\in A}$,
		\begin{align}\label{eq-adj-formula}
			(Q^*\psi)_\alpha = \begin{cases}
				  \dfrac{\|P\alpha\|}{\|\alpha\|}\psi_{P\alpha} & \qquad \text{if } P\alpha\neq 0,\\
				0& \qquad \text{if } P\alpha=0.
			\end{cases}
		\end{align}
		
		Now it will be shown that $\{(s,-us)\,:\,s\in S\}$ is an invariant subspace for $Q\oplus P$. If $s\in S$, then
		\begin{align*}
			P us = \sum_{\beta\in A} \frac{P\beta s_\beta}{\|\beta\|} = \sum_{\alpha\in A} \frac{\alpha }{\|\alpha\|} \sum_{\beta\,:\,\alpha=P\beta} s_\beta \frac{\|P \beta\|}{\|\beta\|}  = \sum_{\alpha\in A} \frac{\alpha (Q s)_\alpha}{\|\alpha\|} = u Qs,
		\end{align*}
		and so 
		\begin{align*}
			(Q\oplus P)(s,-us) = (Qs , -uQs) \in \{(s,-us)\,:\,s\in S \}.
		\end{align*}
		Thus the projection $\tilde{P}$ on $\tilde{E}$ is defined to be the corresponding quotient operator of $Q\oplus P$. Lastly, if $e\in E$, then
		\begin{align}\label{eq-proj-ext}
			\tilde{P} j e = \tilde{P}\pi(0,e) = \pi(0,Pe) = jPe.
		\end{align}
		That is, $\tilde{P}|_{jE}=jP$, and so $\|\tilde{P}\|=\|P\|$.
	\end{proof}
	
	By repeatedly using this lemma, the projection $P$ can be lifted to a projection defined on $X_E$.
	
	\begin{lemma}\label{lem-ind-lim-lift}
		Let $E$ be a Banach space and $P:E\to E$ a projection. Then there is a projection $\hat{P}$ on $X_E$ such that $\|\hat{P}\|=\|P\|$ and $\hat{P}i_ne = i_nPe$ for all $n\in\N$ and $e\in E_n$.
	\end{lemma}

	\begin{proof}
	By inductively using Lemma~\ref{lem-proj-ext} with $P=P_0$, one can construct for each $n\in\N$ a projection $P_{n+1}$ on $E_{n+1}$ such that $\|P_{n+1}\|=\|P\|$ and
	\begin{align}\label{eq-well-defined}
		P_{n+1}|_{j_nE_n} = j_nP_{n}.
	\end{align} 
	Recalling that the maps $i_n$ from \eqref{ind-lim-i_n} satisfy $i_{n+1}j_n=i_n$ for all $n\in\N$, define $\hat{P}$ on $\bigcup_{n=0}^\infty i_n(E_n)$ by
	\begin{align*}
		\hat{P} i_n e = i_n P_{n} e
	\end{align*}
	whenever $e\in E_n$, which is well-defined by \eqref{eq-well-defined}. Observe that $\hat{P}$ is a projection with $\|\hat{P}\|=\|P\|$, which can then be extended to a projection with the same norm, also denoted by $\hat{P}$, on the inductive limit $X_E$ via density.
	\end{proof}
	
	This construction can be used to lift an increasing sequence of projections on $E$ to one on $\tilde{E}$, and hence also to $X_E$.

	\begin{lemma}\label{lem-proj-ext-seq}
	Suppose that $E$ is a Banach space and $\{P_k\}_{k=1}^\infty$ is an increasing sequence of uniformly bounded projections on $E$ with $\sup\limits_k\|P_k\| = K$. Then there is an increasing sequence of projections $\{\tilde{P}_k\}_{k=1}^\infty$ on $\tilde{E}$ with $\sup\limits_k\|\tilde{P}_k\|= K$. Moreover, $\tilde{P}_k|_{jE} = jP_k$ for all $k\in\Z^+$.
\end{lemma}

\begin{proof}
	Suppose that $\{P_k\}_{k=1}^\infty$ is an increasing sequence of uniformly bounded projections on a Banach space $E$. For each $k\in \Z^+$, let $\tilde{P}_k$ be the projection obtained from $P_k$ and Lemma~\ref{lem-proj-ext}, so that $\tilde{P}_k$ is the quotient operator induced by $Q_k\oplus P_k$. Then it remains to show that $\{\tilde{P}_k\}_{k=1}^\infty$ is increasing, which would follow from $\{Q_k\}_{k=1}^\infty$ being increasing. From \eqref{eq-adj-formula}, one has that
	\begin{align*}
		(Q_k^*\psi)_\alpha = \begin{cases}
			\dfrac{\|P_k\alpha\|}{\|\alpha\|} \psi_{P_k\alpha}   & \text{ if } P_k\alpha \neq 0,\\
			0 & \text{ if } P_k\alpha =0
		\end{cases}
	\end{align*}
	for each $k\in\Z^+$. Now suppose that $\ell> k$ and $\psi\in B^*$. If $P_k \alpha=0$, then it easily follows that $(Q_k^*Q_\ell^*\psi)_\alpha = (Q_\ell^*Q_k^*\psi)_\alpha=0 = (Q_k^*\psi)_\alpha$. If $P_k\alpha\neq 0$, then
	\begin{align*}
		(Q_{k}^* Q_{\ell}^* \psi)_\alpha = \frac{\|P_{k}\alpha\|}{\|\alpha\|} (Q_{\ell}^*\psi)_{P_{k}\alpha} = \dfrac{\|P_{k}\alpha\|}{\|\alpha\|} \frac{\|P_{\ell}P_{k}\alpha\|}{\|P_{k}\alpha\|} \psi_{P_{\ell}P_{k}\alpha} = \frac{\|P_{k}\alpha\|}{\|\alpha\|} \psi_{P_{k}\alpha} = (Q_{k}^*\psi)_\alpha
	\end{align*}
	and
	\begin{align*}
		(Q_{\ell}^* Q_{k}^* \psi)_\alpha = \frac{\|P_{\ell}\alpha\|}{\|\alpha\|} (Q_{k}^*\psi)_{P_{\ell}\alpha} = \frac{\|P_{\ell}\alpha\|}{\|\alpha\|} \frac{\|P_{k}P_{\ell}\alpha\|}{\|P_{\ell}\alpha\|} \psi_{P_{k}P_{\ell}\alpha} = \frac{\|P_{k}\alpha\|}{\|\alpha\|} \psi_{P_{k}\alpha} = (Q_{k}^*\psi)_\alpha.
	\end{align*}
	Moreover, $Q_\ell\neq Q_k$ since there are operators $\alpha\in A$ for which $P_\ell\alpha\neq0$ but $P_k\alpha=0$. So $\{Q_k^*\}_{k=1}^\infty$, and hence $\{Q_k\}_{k=1}^\infty$, are increasing sequences of uniformly bounded projections. Hence $\{\tilde{P}_k\}_{k=1}^\infty$ is an increasing sequence of uniformly bounded projections on $\tilde{E}$.	
\end{proof}

\begin{theorem}\label{thm-ind-lim-lift}
Let $E$ be a Banach space and $\{P_k\}_{k=1}^\infty$ an increasing sequence of uniformly bounded projections on $E$. Then there is an increasing sequence of uniformly bounded projections $\{\hat{P}_k\}_{k=1}^\infty$ on $X_E$ with $\sup\limits_k\|\hat{P}_k\| = \sup\limits_k\|P_k\|$.
\end{theorem}
	
\begin{proof}
		Suppose that $\{P_k\}_{k=1}^\infty$ is an increasing sequence of uniformly bounded projections on a Banach space $E$.  Lemmas~\ref{lem-ind-lim-lift} and \ref{lem-proj-ext-seq} yield an increasing sequence of projections $\{\hat{P}_k\}_{k=1}^\infty$ on $X_E$ with $\sup\limits_k\|\hat{P}_k\| = \sup\limits_k\|P_k\|$. 
	\end{proof}

	\subsection{Main Result}
	
	In this section, it will be shown that no matter the Banach space $E$, there is a well-bounded operator on $X_E$ that is not of type (B).

	\begin{lemma}\label{lem-SOT}
		Suppose that $\{P_k\}_{k=1}^\infty$ and $\{\tilde{P}_k\}_{k=1}^\infty$ are as in Lemma~\ref{lem-proj-ext-seq}. Then $\lim\limits_{k\to\infty} \tilde{P}_k$ does not exist in the SOT.
	\end{lemma}
	
	\begin{proof}
		As $\{P_k\}_{k=1}^\infty$ is an increasing sequence of projections, let $\{e_k\}_{k=1}^\infty$ be a set of unit vectors in $E$ satisfying $e_1\in P_1E$ and $e_k\in P_kE\cap\ker P_{k-1}$ for all $k\geq 2$. Further let  $e=\sum\limits_{k=1}^\infty 2^{-k} e_k\in E$ and $\gamma\in A$ be a rank 1 operator whose range is $\text{span}\{e\}$ with $\|\gamma\|=1$, so that $P_{\ell}\gamma\neq P_k \gamma$ for all $\ell\neq k$. Let $b\in B_\gamma\backslash S_\gamma$ with $d(b,S_\gamma)=1$, and $x=(x_\alpha)_{\alpha\in A} \in B$ be defined by
		\begin{align*}
			x_\alpha = \begin{cases}
				b & \qquad \text{ if } \alpha=\gamma,\\
				0 & \qquad \text{ otherwise}. 
			\end{cases}
		\end{align*}
		We will show that $\{\tilde{P}_k \pi(x,0)\}_{k=1}^\infty$ does not converge. This is essentially because for $\ell\neq k$, $P_\ell\gamma\neq P_k\gamma$ means that $b$ will always occupy a different coordinate in $Q_\ell x$ to $Q_k x$.
				So, if $s\in S$ and $\ell>k$, we have that
				\begin{align*}
					\big\|  (Q_{\ell}\oplus P_\ell - Q_{k}\oplus P_k) (x-s , us)   \big\|  & \geq \| (Q_{\ell}-Q_{k}) (x-s) \| \\
					&	= \sum_{\alpha\in A} \bigg\| \sum_{\beta\,:\, \alpha=P_{\ell}\beta} (x-s)_\beta \frac{\|P_{\ell}\beta\|}{\|\beta\|} - \sum_{\beta\,:\,\alpha=P_{k}\beta} (x-s)_\beta \frac{\|P_{k}\beta\|}{\|\beta\|} \bigg\|.
				\end{align*}
				By considering only the $\alpha=P_{\ell}\gamma$ coordinate, we have that
				\begin{align*}
					\| (Q_{\ell}-Q_{k}) (x-s) \|  &\geq \bigg\|x_\gamma \|P_{\ell}\gamma\| - \sum_{\beta\,:\, P_{\ell}\gamma = P_{\ell}\beta} s_\beta\frac{\|P_{\ell}\beta\|}{\|\beta\|}  + \sum_{\beta\,:\, P_{\ell}\gamma = P_{k}\beta} s_\beta\frac{\|P_{k}\beta\|}{\|\beta\|}  \bigg\|  \\
					&\geq  \|P_{\ell}\gamma\| \, d(b,S_\gamma)\\
					& = \|P_\ell\gamma\|.
				\end{align*}
				Since $P_{n}e \to e$ as $n\to\infty$, there is some $N\in\Z^+$ such that $2\|P_{n}e\| > \|e\|$ whenever $n>N$. Let $s_e\in S_\gamma$ satisfy $\gamma s_e = e$ so that
				\begin{align*}
					\|P_{n}\gamma\| \geq \frac{\| P_{n}\gamma s_e\|}{\|s_e\|} = \frac{\|P_{n }e\|}{\|s_e\|} > \frac{\|e\|}{2\|s_e\|}
				\end{align*} 
				for all $n>N$. Hence
				\begin{align*}
					\| (\tilde{P}_{\ell}  - \tilde{P}_{k})\pi(x,0) \| \geq \inf_{s\in S} \| (Q_{\ell}-Q_{k}) (x-s) \| \geq \frac{\|e\|}{2\|s_e\|}
				\end{align*}
				for all $\ell>k>N$. That is, $\{\tilde{P}_{k} \}_{k=1}^\infty$ does not converge in the SOT.
			\end{proof}
			
			Consequently, whenever a Banach space $E$ has an increasing sequence of uniformly bounded projections, Lemma~\ref{lem-ind-lim-lift} and Theorem~\ref{thm-proj-B} yield a well-bounded operator on $X_E$ not of type (B). Now, it is not known in general whether one can always find such a sequence of projections on an infinite dimensional Banach space $E$, and certainly it is not possible when $E$ is finite dimensional. However, it will be shown that one can always find such a sequence of projections on $\tilde{E}$.

			\begin{lemma}\label{lem-fin-dim}
				If $E$ is finite dimensional, then there is an increasing sequence of uniformly bounded projections on $\tilde{E}$.
			\end{lemma}
		
			\begin{proof}
				As $E$ is finite dimensional, $jE$ is complemented in $\tilde{E}$ and so one makes the identification
			\begin{align*}
				\tilde{E} = jE \oplus (\tilde{E}/jE)
			\end{align*}
			as a topological direct sum. Recall that $\tilde{E}/jE$ is isometric to $B/S$, which can be identified isometrically by Lemma~\ref{lem-quotient isom} as
		$ \ell^1 (B_\alpha/S_\alpha)_{\alpha\in A}$.
			As $A$ is infinite, one can easily construct an increasing sequence of uniformly bounded projections on $\tilde{E}/jE$, and hence also on $\tilde{E}$.
			\end{proof}

			For the case of infinite dimensional $E$, one can construct a uniformly bounded sequence of increasing projections on $\tilde{E}$ in the following way. Recall from the construction of $\tilde{E}$ that $B_2=L^1(\T)$ and $S_2=H^1(\T)$. Define the projection $p:B_2\to B_2$, whose range is in $S_2$, by 
			\begin{align*}
				p(f) = \frac{1}{2\pi}\int_0^{2\pi} f(e^{i\theta})\,d\theta.
			\end{align*} 
			That is, $p(f)$ is the constant function whose value is the mean of $f$. Now let $\{e_m\}_{m=1}^\infty$ be a set of linearly independent elements of the unit sphere of $E$. For each $m\in\Z^+$, define $\alpha_m\in F(S_2,E)$ by
			\begin{align*}
				\alpha_m(f) = \left( \frac{1}{2\pi} \int_0^{2\pi}f(e^{i\theta})(1+e^{-i\theta})\,d\theta\right) e_m.
			\end{align*}
			As $\{e_m\}_{m=1}^\infty$ is linearly independent, we have that $\alpha_m p\neq \alpha_n p$ and $\alpha_m\neq\alpha_n$, for all $m\neq n$. Also, for all $m\in\Z^+$, one can easily verify that $\alpha_m\neq\alpha_m p$,  $\|\alpha_m\|\leq 2$ and $\|\alpha_m p\|= 1$.  Recall that $A=\bigcup_{i=1}^3 F(S_i,E)$ and $B=\ell^1(B_\alpha)_{\alpha\in A}$. For each $n\in\Z^+$, define $Q_n:B\to B$ by
			\begin{align*}
				(Q_nx)_\alpha = \begin{cases}
					x_\alpha & \qquad \text{if } \alpha=\alpha_m p \;\text{ for some } 1\leq m\leq n,\\
					- \|\alpha_m\| px_{\alpha_m p} & \qquad \text{if } \alpha = \alpha_m \hspace*{2mm}\; \text{ for some } 1\leq m\leq n,\\
					0 & \qquad \text{otherwise.} 
				\end{cases}
			\end{align*} 
			
			Elaborating upon this definition, for each $1\leq m\leq n$, the map $Q_n$ preserves the $(\alpha_mp)$th coordinate, replaces the $(\alpha_m)$th coordinate by $-\|\alpha_m\|px_{\alpha_mp}$, and deletes all the other coordinates.
			
			\begin{lemma}\label{lem-Bprojections}
				$\{Q_n\}_{n=1}^\infty$ is an increasing sequence of uniformly bounded projections on $B$.
			\end{lemma}
			
			\begin{proof}
				Clearly each $Q_n$ is linear and $\sup\limits_n\|Q_n\| \leq 3$. Now, for $k\geq n$ and $\alpha\in A\backslash\{\alpha_m\}_{m=0}^n$, one has that $(Q_n Q_k x)_\alpha = (Q_k Q_n x)_\alpha = (Q_n x)_\alpha$ as these are all of the coordinates of $x$ that are unchanged or deleted. Otherwise, if $\alpha=\alpha_m$ for some $1\leq m\leq n$, we have that
				\begin{align*}
					(Q_kQ_nx)_{\alpha_m} & = -\|\alpha_m\|p(Q_nx)_{\alpha_m p} = -\|\alpha_m\|p x_{\alpha_m p}	= (Q_nx)_{\alpha_m}
				\end{align*}
				and
				\begin{align*}
		(Q_nQ_kx)_{\alpha_m} & =  -\|\alpha_m\|p(Q_kx)_{\alpha_m p} = -\|\alpha_m\|p x_{\alpha_m p} = (Q_nx)_{\alpha_m}.
				\end{align*}
				Thus we have that $Q_kQ_n=Q_nQ_k=Q_n$ whenever $k\geq n$. Moreover, $Q_k\neq Q_n$ for all $k\neq n$, and so $\{Q_n\}_{n=1}^\infty$ is an increasing sequence of uniformly bounded projections.
		
			\end{proof}
		
			\begin{remark}
				One can construct $\{Q_n\}_{n=1}^\infty$ more generally with any operator $p\in B(X_i,S_i)$ and collection of operators $\{\alpha_m\}_{m=1}^\infty\subset F(S_i,E)$ satisfying $\sup\limits_m\|\alpha_m\|<\infty$, $\alpha_m p\neq \alpha_m $ for all $m\in\Z^+$, and $\alpha_m p\neq\alpha_n p$ for all $m\neq n$. 
			\end{remark}

			The projections in Lemma~\ref{lem-Bprojections} may now be used to defined an increasing sequence of uniformly bounded projections on $\tilde{E}$.
			
			\begin{lemma}\label{lem-inf-dim}
				If $E$ is an infinite dimensional Banach space, then there is an increasing sequence of uniformly bounded projections $\{P_n\}_{n=1}^\infty$ on $\tilde{E}$.
			\end{lemma}
			
			\begin{proof}
				Let $\{Q_n\}_{n=1}^\infty$ be as in Lemma~\ref{lem-Bprojections} and consider the increasing sequence of uniformly bounded projections $\{Q_n\oplus 0\}_{n=1}^\infty$ on $B\oplus E$. It will be shown that $\{(s,-us)\}_{s\in S}$ is invariant under $Q_n\oplus 0$ for each $n\in\Z^+$. 
				Recalling that $\|\alpha_m p\| = 1$ for all $m\in\N$, we have that
				\begin{align*}
					u Q_n s & = \sum_{\alpha\in A} \frac{\alpha (Q_ns)_\alpha}{\|\alpha\|} = \sum_{m=0}^n \frac{\alpha_m p s_{\alpha_m p}}{\|\alpha_m p\|} - \sum_{m=0}^n \frac{\alpha_m (\|\alpha_m\|ps_{\alpha_m p} ) }{\|\alpha_m\|} = 0,
				\end{align*}
				and so $(Q_n\oplus 0)(s,-us) = (Q_ns,0) = (Q_ns, -uQ_ns)$ for all $s\in S$. Thus setting $P_n$ to be the quotient of $Q_n$ for each $n\in\Z^+$, which will satisfy $P_n\neq P_m$ for all $n\neq m$, defines an increasing sequence of uniformly bounded projections $\{P_n\}_{n=1}^\infty$ on $\tilde{E}$ . 
			\end{proof}
		
			\begin{corollary}\label{cor-inc-proj}
				For any Banach space $E$, there is an increasing sequence of uniformly bounded projections on $X_E$.
			\end{corollary}
		
			\begin{proof}
				Suppose that $E$ is a Banach space. By Lemmas~\ref{lem-fin-dim} and \ref{lem-inf-dim}, there is a uniformly bounded sequence of increasing projections $\{P_n\}_{n=1}^\infty$ on $\tilde{E}$. As $X_E=X_{\tilde{E}}$, Lemma~\ref{lem-ind-lim-lift} yields a uniformly bounded sequence of increasing projections on $X_E$.
			\end{proof}
			
			The main result can now be easily concluded.
			
			\begin{theorem}
				Let $E$ be a Banach space. Then there is a well-bounded operator on $X_E$ that is not of type (B).
			\end{theorem}
			
			\begin{proof}
				By Corollary~\ref{cor-inc-proj}, there is a uniformly bounded sequence of increasing projections $\{P_n\}_{n=1}^\infty$ on $X_E$ for any Banach space $E$. By Lemma~\ref{lem-SOT}, $\lim\limits_{n\to\infty}P_n$ does not exist in the SOT, and hence the well-bounded operator on $X_E$ obtained via Theorem~\ref{thm-proj-B} is not of type (B).
			\end{proof}

\section*{Acknowledgements}
The work of the author was supported by the Research Training Program of the Department of Education and Training of the Australian Government. The author would also like to thank Ian Doust for checking the manuscript.

\end{document}